\documentclass[12pt,reqno]{amsart}

\usepackage{amssymb}
\usepackage{hyperref}

\input xy
\xyoption{all}

\newtheorem{theorem}{Theorem}[section]
\newtheorem{proposition}[theorem]{Proposition}
\newtheorem{lemma}[theorem]{Lemma}

\DeclareMathOperator{\End}{End}
\DeclareMathOperator{\PE}{PEnd}
\DeclareMathOperator{\SPE}{SPEnd}

\numberwithin{equation}{section}

\begin{document}
\baselineskip=15.5pt

\title[Bohr-Sommerfeld Lagrangians of moduli of Higgs
bundles]{Bohr--Sommerfeld Lagrangians of moduli spaces of Higgs bundles}

\author[I. Biswas]{Indranil Biswas}

\address{School of Mathematics, Tata Institute of Fundamental
Research, Homi Bhabha Road, Bombay 400005, India}

\email{indranil@math.tifr.res.in}

\author[N. L. Gammelgaard]{Niels Leth Gammelgaard}

\address{Center for Quantum Geometry of Moduli Spaces, University of
Aarhus, Ny Munkegade 118, DK-8000 Aarhus, Denmark}

\email{nlg@qgm.au.dk}

\author[M. Logares]{Marina Logares}

\address{Instituto de Ciencias Matem\'aticas (CSIC-UAM-UC3M-UCM), C/
Nicolas Cabrera 15, 28049 Madrid, Spain}

\email{marina.logares@icmat.es}



\date{}

\begin{abstract}
Let $X$ be a compact connected Riemann surface of genus at least
two. Let $M_H(r,d)$ denote the moduli space of semistable Higgs
bundles on $X$ of rank $r$ and degree $d$. We prove that the compact
complex Bohr--Sommerfeld Lagrangians of $M_H(r,d)$ are precisely the
irreducible components of the nilpotent cone in $M_H(r,d)$. This
generalizes to Higgs $G$--bundles and also to the parabolic Higgs
bundles.
\end{abstract}

\maketitle

\section{Introduction}\label{sec1}

Let $M_H(r,d)$ be the moduli space of semistable Higgs bundles of rank
$r$ and degree $d$ on a compact connected Riemann surface $X$ of genus
$g$ at least two. It is an irreducible normal complex projective
variety of complex dimension $2(r^2(g-1)+1)$. This moduli space is
equipped with an algebraic symplectic form. In fact, there is a
canonical algebraic one-form $\omega$ on $M_H(r,d)$ such that
$d\omega$ is the symplectic form. Let
$$
{\mathcal L}\, :=\, {\mathcal O}_{M_H(r,d)}\,=\, M_H(r,d)\times
\mathbb C
$$
be the trivial holomorphic line bundle on $M_H(r,d)$. Consider the
holomorphic connection
$$
D\,:=\, d+\omega
$$
on ${\mathcal L}$, where $d$ denotes the de Rham differential on
functions on $M_H(r,d)$. We note that the curvature of $D$ is
$d\omega$.

A compact Lagrangian on $M_H(r,d)$ is a reduced irreducible compact
complex analytic subset
$$
{\mathbb L} \,\subset\, M_H(r,d)
$$
of dimension $(\dim M_H(r,d))/2\,=\, r^2(g-1)+1$ such that the
restriction of $({\mathcal L}\, ,D)$ to ${\mathbb L}$ is a flat line
bundle. Indeed, this follows immediately from the fact that the curvature
of $D$ coincides with the symplectic form on $M_H(r,d)$.
A compact Bohr--Sommerfeld Lagrangian on $M_H(r,d)$ is a
reduced irreducible compact complex analytic subset ${\mathbb L}
\,\subset\, M_H(r,d)$ of dimension $r^2(g-1)+1$ such that
${\mathcal L}$ admits a nonzero flat section over ${\mathbb L}$.
Clearly, a Bohr--Sommerfeld Lagrangian is a Lagrangian.

A flat section of ${\mathcal L}$ over ${\mathbb L}$ gives a
holomorphic function on ${\mathbb L}$ because the form $\omega$ is
holomorphic, and ${\mathcal L}$ coincides with
${\mathcal O}_{M_H(r,d)}$. On the other hand, $\mathbb L$ does not admit any
nonconstant holomorphic function because it is compact and irreducible. A
constant function is a flat section of ${\mathcal L}\vert_{\mathbb L}$
if and only if the pullback of the connection $D$ to ${\mathcal
L}\vert_{\mathbb L}$ is the de Rham differential. Consequently, a
reduced irreducible compact complex analytic subset
$$
\iota\, :\, {\mathbb L}\, \hookrightarrow \, M_H(r,d)
$$
is a Bohr--Sommerfeld Lagrangian if and only if $\iota^*\omega\,=\,
0$.

The Hitchin map
$$
{\mathcal H}\, :\, M_H(r,d)\, \longrightarrow\, \bigoplus_{i=1}^r
H^0(X,\, K^{\otimes i}_X)
$$
sends any $(E\, ,\theta)$ to $\sum_{i=1}^r
\text{trace}(\theta^i)$. The fiber of ${\mathcal H}$ over $0$ is the
nilpotent cone.

We prove the following:

\begin{theorem}\label{thm0}
The compact Bohr--Sommerfeld Lagrangians in $M_H(r,d)$ are precisely
the irreducible components of the nilpotent cone.
\end{theorem}

\noindent Theorem \ref{thm0} generalizes to other contexts; see Section
\ref{3.1} and Section \ref{3.2}.

Quantization is a procedure which associates to a symplectic manifold $M$ a Hilbert 
space $Q(M)$.  When using geometric quantization, the quantum space is constructed 
from sections or higher cohomology groups of an appropriate complex line bundle, known 
as the prequantization bundle, equipped with a connection whose 
curvature equals the symplectic form. The interesting sections of the prequantization 
bundle are those provided by a choose of a polarization of $M$. For real polarizations 
a theorem by Sniaticky \cite{S} determines the dimension of the quantized space in 
terms of the number of Bohr-Sommerfeld Lagrangians for the symplectic form of $M$.

In the context of moduli spaces, Bohr-Sommerfeld Lagrangians have
previously been studied by Jeffrey and Weitsman \cite{MR1204322} for a
certain real polarization, on the moduli space of flat
$\mathrm{SU}(2)$-connections on a smooth two-manifold, associated to a
pair-of-pants decomposition. Quantizing the moduli space using this
polarization results in distributional quantum states supported on its
Bohr-Sommerfeld Lagrangians. Andersen proved that the Jeffrey-Weitsman
polarization is obtained as the limit of the K\"ahler polarization on
the moduli space of stable rank two bundles with trivial determinant on
a Riemann surface by pinching the Riemann surface structure along the
curves of the pair-of-pants decomposition \cite{MR1612326}. He used
this fact, and the Bohr-Sommerfeld description of the quantum space
for the Jeffrey-Weitsman polarization, to construct a mapping class
group invariant unitary structure for the Hitchin connection in this
setting, he also gives a geometric formula for the Witten-Reshetikhin-Turaev
quantum invariants and proves that the colored Jones polynomials detect the
unknot \cite{A5,1206.2785}.

\section{One-form on the moduli space}

Let $X$ be a compact connected Riemann surface of genus $g$, with $g\,
\geq\, 2$. The holomorphic cotangent bundle of $X$ will be denoted by
$K_X$. Fix an integer $d$ and also fix a positive integer $r$. Let
$M_H(r,d)$ denote the moduli space of semistable Higgs bundles on $X$
of rank $r$ and degree $d$. The space $M_H(r,d)$ is an irreducible normal
complex projective variety of complex dimension $2(r^2(g-1)+1)$.

There is a natural algebraic $1$-form on $M_H(r,d)$; we will recall
its construction. Take any Higgs bundle $(E\, ,\theta)$. Define
$$
C^0\, :=\, \End(E) \,=\, E\otimes E^\vee ~\ \text{ and }~\ C^1\,
:=\, \End(E)\otimes K_X\, .
$$
Let $\text{ad}(\theta)\, :\, C^0\, \longrightarrow\, C^1$ be the
${\mathcal O}_X$--linear homomorphism
defined by
$$
s\, \longmapsto\, [s\, ,\theta]\,=\, s\circ \theta-\theta\circ s\, ,
$$
where the composition is the usual composition of endomorphisms. Now consider
the two-term complex
\begin{equation}\label{e1}
{\mathcal C}^\bullet \, :\, C^0\, \stackrel{\text{ad}(\theta)}{\longrightarrow}\, C^1\, .
\end{equation}
The infinitesimal deformations of the Higgs bundle $(E\, ,\theta)$ are parametrized
by the first hypercohomology ${\mathbb H}^1({\mathcal C}^\bullet)$ \cite{Hi2},
\cite{BR}, \cite{Bot}, \cite{Ma}.

The natural homomorphism of complexes
\begin{equation}\label{co}
\xymatrix{
C^{0}\ar[r]\ar@{=}[d] & C^{1}\ar[d]\\
\End(E)\ar[r] & 0
}
\end{equation}
induces a homomorphism
\begin{equation}\label{e2}
q\, :\, {\mathbb H}^1({\mathcal C}^\bullet)\, \longrightarrow\, H^1(X,\, C^0)\,=\,
H^1(X,\, \End(E))\, .
\end{equation}
We note that $H^1(X,\, \End(E))$ parametrizes the infinitesimal deformations of the
holomorphic vector bundle $E$. The above homomorphism $q$ coincides with the
forgetful map that sends an infinitesimal deformation of the Higgs bundle
$(E\, ,\theta)$ to the infinitesimal deformation of the underlying vector bundle $E$
(forgetting the Higgs field $\theta$). Let
$$
{\mathcal S}{\mathcal D}\, :\, H^1(X,\, \End(E))
\otimes H^0(X,\, \End(E)\otimes K_X)\, \longrightarrow\, H^1(X,\, K_X)\,=\, \mathbb C
$$
be the pairing given by Serre duality. Define
\begin{equation}\label{e3}
\omega(E,\theta)\, :\, {\mathbb H}^1({\mathcal C}^\bullet)\,\longrightarrow\,
{\mathbb C}\, ,
~ \ \alpha\, \longmapsto\, {\mathcal S}{\mathcal D}(q(\alpha)\otimes\theta)\, ,
\end{equation}
where $q$ is the homomorphism in \eqref{e2}.
The one-form $\omega$ on $M_H(r,d)$ defined by $(E\, ,\theta)\, \longmapsto\,
\omega(E,\theta)$ is algebraic. The algebraic two--form $d\omega$ is the symplectic
form on $M_H(r,d)$ \cite{Hi1}, \cite{Hi2}, \cite{BR}, \cite{Bot}, \cite{Ma}.

Let $N^s(r,d)$ denote the moduli space of stable vector bundles on $X$ of rank $r$ and
degree $d$. The total space $T^*N^s(r,d)$ of the (algebraic) cotangent bundle of
$N^s(r,d)$ is a Zariski open subset of $M_H(r,d)$. The restriction of $\omega$ to
$T^*N^s(r,d)$ coincides with the Liouville one-form on $T^*N^s(r,d)$.

For any nonzero complex number $\lambda$, let
\begin{equation}\label{tl}
T_\lambda\, :\, M_H(r,d)\, \longrightarrow\, M_H(r,d)\, , \ ~ (E\, ,\theta)\,\longmapsto
\, (E\, ,\lambda\cdot \theta)
\end{equation}
be the automorphism. Clearly these automorphisms together define an action of the
multiplicative group ${\mathbb C}^*\,=\,{\mathbb C}\setminus\{0\}$ on $M_H(r,d)$. Let
\begin{equation}\label{xi}
\xi\, \in\, H^0(M_H(r,d),\, TM_H(r,d))
\end{equation}
be the algebraic vector field associated to this action of ${\mathbb C}^*$.

\begin{lemma}\label{lem1}
The one-form $i_\xi d\omega$ obtained by contracting the symplectic form $d\omega$ using
the vector field $\xi$ satisfies the equation
$$
i_\xi d\omega\,=\, \omega\, .
$$
\end{lemma}

\begin{proof}
We first show that $q(\xi)\,=\, 0$, where $q$ is the homomorphism in \eqref{e2}.
This follows from the fact that the action of ${\mathbb C}^*$ does not change the
underlying vector bundle (recall that $q$ coincides with the forgetful map
that sends the infinitesimal deformations of a Higgs bundle
to the infinitesimals deformation of the underlying vector bundle). To prove
this directly, first note that the kernel of the homomorphism of complexes in
\eqref{co} is the complex
$$
\widetilde{\mathcal C}^\bullet \, :\, \widetilde{C}^0\,:=\, 0\, \longrightarrow\,
\widetilde{C}^1\,:=\, \End(E)\otimes K_X\, .
$$
The tangent vector $\xi((E, \theta))\,\in\, T_{(E, \theta)}M_H(r,d)\,=\,
{\mathbb H}^1({\mathcal C}^\bullet)$ lies in the image of of the homomorphism
$$
H^0(X,\,  \End(E)\otimes K_X)\,=\,
{\mathbb H}^1(\widetilde{\mathcal C}^\bullet)\,\longrightarrow\,
{\mathbb H}^1({\mathcal C}^\bullet)
$$
given by the natural homomorphism $\widetilde{\mathcal C}^\bullet\,\longrightarrow\,
{\mathcal C}^\bullet$. In fact, $\xi((E, \theta))$ is the image of
$$
\theta\, \in\, H^0(X,\,  \End(E)\otimes K_X)\,=\,
{\mathbb H}^1(\widetilde{\mathcal C}^\bullet)\, .
$$
This immediately implies that $q(\xi)\,=\, 0$.

Since $q(\xi)\,=\, 0$, from the definitions of $\omega$ it follows that
\begin{equation}\label{f1}
\omega(\xi)\, =\, 0\, .
\end{equation}
It also
follows from the definitions of $\omega$ and $T_\lambda$ (see
\eqref{tl}) that $T_\lambda^*\omega\,=\,\lambda\cdot \omega$. Therefore, we have
\begin{equation}\label{f2}
L_\xi\omega \, =\, \omega\, .
\end{equation}
Consider the identity $L_\xi\omega \, =\, i_\xi d\omega + d(\omega(\xi))$. In view
of \eqref{f1} and \eqref{f2}, the lemma follows from this identity.
\end{proof}

\section{Bohr--Sommerfeld Lagrangians on the moduli space}

Consider the symplectic from $d\omega$ on $M_H(r,d)$, where $\omega$ is constructed
in \eqref{e3}. A \textit{compact complex Lagrangian} on $M_H(r,d)$ is a reduced
irreducible compact complex analytic subset
$$
\iota\, :\, {\mathbb L}\, \hookrightarrow \, M_H(r,d)
$$
of dimension $r^2(g-1)+1$ such that $\iota^*d\omega\,=\, 0$. A \textit{compact
Bohr--Sommerfeld Lagrangian} on $M_H(r,d)$ is a reduced irreducible compact complex
analytic subset $\iota
\, :\, {\mathbb L}\, \hookrightarrow \, M_H(r,d)$ of dimension $r^2(g-1)+1$ such that
$$
\iota^*\omega\,=\, 0\, .
$$
Since $\iota^*d\omega\,=\, d\iota^*\omega$, a {Bohr--Sommerfeld
Lagrangian} is indeed a Lagrangian.

\begin{lemma}\label{lem2}
Let $\iota\, :\, {\mathbb L}\, \hookrightarrow \, M_H(r,d)$ be a Bohr--Sommerfeld
Lagrangian. For any smooth point $z\, \in\, {\mathbb L}$, the tangent vector
$$\xi (z)\, \in\, T_z M_H(r,d)$$ lies in the subspace $T_z {\mathbb L}\,\subset\,
T_z M_H(r,d)$, where $\xi$ is the vector field in \eqref{xi}.
\end{lemma}

\begin{proof}
Since ${\mathbb L}$ is a Bohr--Sommerfeld Lagrangian, we have
$$
\omega (v)\, =\, 0\, , ~ \ \forall ~ \ v\, \in\, T_z {\mathbb L}\, .
$$
Hence from Lemma \ref{lem1} it follows that
\begin{equation}\label{ze}
d\omega (\xi(z)\, , v)\, =\, 0\, , ~ \ \forall ~ \ v\, \in\, T_z {\mathbb L}\, .
\end{equation}
Since $\mathbb L$ is Lagrangian for the symplectic form $d\omega$, from \eqref{ze}
it follows immediately that $\xi(z)\,\in\, T_z {\mathbb L}$.
\end{proof}

Let
\begin{equation}\label{h}
{\mathcal H}\, :\, M_H(r,d)\, \longrightarrow\, {\mathcal V}\,:=\, \bigoplus_{i=1}^r
H^0(X,\, K^{\otimes i}_X)\, , \ ~ (E\, ,\theta)\,\longmapsto \, \sum_{i=1}^r
\text{trace}(\theta^i)
\end{equation}
be the Hitchin map. Every irreducible component of every fiber of $\mathcal H$ is a
compact complex Lagrangian \cite{Hi2}. The fiber
$$
{\mathcal N}\, :=\, {\mathcal H}^{-1}(0)\, \subset\, M_H(r,d)
$$
is known as the nilpotent cone.

\begin{proposition}\label{prop1}
All compact Bohr--Sommerfeld Lagrangians in $M_H(r,d)$ are contained in the
nilpotent cone $\mathcal N$.
\end{proposition}

\begin{proof}
For any nonzero complex number $\lambda$, consider the linear automorphism of
the vector space ${\mathcal V}$ in \eqref{h} defined by
$$
(c_1\, ,\cdots\, , c_i\, ,\cdots\, , c_r)\, \longmapsto\,
(\lambda\cdot c_1\, ,\cdots\, , \lambda^i\cdot c_i\, ,\cdots\, ,\lambda^r\cdot c_r)\, .
$$
These automorphisms together define an action of ${\mathbb C}^*$ on ${\mathcal V}$. The
morphism $\mathcal H$ in \eqref{h} is clearly ${\mathbb C}^*$--equivariant for this
action of ${\mathbb C}^*$ on ${\mathcal V}$ and the action of ${\mathbb C}^*$ on
$M_H(r,d)$ in \eqref{tl}.

Let
$$
\iota\, :\, {\mathbb L}\, \hookrightarrow \, M_H(r,d)
$$
be a compact Bohr--Sommerfeld Lagrangian. So ${\mathbb L}$ does not admit any
nonconstant holomorphic map to $\mathcal V$. In particular, ${\mathcal H}\circ \iota$
is a constant map, where $\mathcal H$ is the Hitchin map defined in \eqref{h}. So we have
\begin{equation}\label{t0}
\iota ({\mathbb L})\, \subset\, {\mathcal H}^{-1}(t_0)
\end{equation}
for some point $t_0\,\in\, \mathcal V$.

{}From Lemma \ref{lem2} it follows that $\iota ({\mathbb L})$ is
preserved by the action of ${\mathbb C}^*$ on $M_H(r,d)$ in \eqref{tl}.
Since the map $\mathcal H$ is ${\mathbb C}^*$--equivariant, we conclude that
$t_0$ in \eqref{t0} is fixed by the action of ${\mathbb C}^*$ on $\mathcal V$.
This implies that $t_0\,=\, 0$.
\end{proof}

\begin{theorem}\label{thm1}
The compact Bohr--Sommerfeld Lagrangians in $M_H(r,d)$ are precisely the
irreducible components of the nilpotent cone $\mathcal N$.
\end{theorem}

\begin{proof}
In view of Proposition \ref{prop1}, it suffices to show each
irreducible component of $\mathcal N$ is indeed a Bohr--Sommerfeld Lagrangian. Let
$$
{\mathbb L}\, \subset \, \mathcal N
$$
be an irreducible component. We note that ${\mathbb L}$ is compact because
$\mathcal H$ is a proper morphism \cite{Hi1}, \cite{Hi2}, \cite{Ni}.

Take a smooth point
$$
z\,=\, (E\, ,\theta)\, \in\, {\mathbb L}\, .
$$
Take any tangent vector
$v\, \in\, T_z {\mathbb L}$. We need to show that
\begin{equation}\label{tbs}
\omega(z)(v)\,=\, 0\, .
\end{equation}
{}From Lemma \ref{lem1} we have
\begin{equation}\label{tbs2}
\omega(z)(v)\,=\, d\omega(\xi\, ,v)\, .
\end{equation}

The subvariety $\mathbb L\, \subset\, M_H(r,d)$ is closed under the action of
${\mathbb C}^*$ on $M_H(r,d)$. Indeed, the Hitchin map ${\mathcal H}$ in \eqref{h} is
${\mathbb C}^*$--equivariant, and the point $0\, \in\, \mathcal V$ is fixed by the
action of ${\mathbb C}^*$. Therefore, it follows that
$$
\xi(z)\, \in\, T_z {\mathbb L}\, .
$$
Since ${\mathbb L}$ is a Lagrangian subvariety of $M_H(r,d)$ for the
symplectic form $d\omega$ on $M_H(r,d)$ \cite{Hi2}, \cite{La}, and
$\xi(z)\, ,v\, \in\, T_z {\mathbb L}$, we have
\begin{equation}\label{tbs3}
d\omega(\xi(z)\, ,v)\,=\, 0\, .
\end{equation}
Finally, \eqref{tbs} follows from \eqref{tbs2} and \eqref{tbs3}.

We give an alternative proof which may be of independent interest.

Associated to
${\mathbb L}$ there is an integer $n\, \geq\, 2$ and pairs of integers $(r_1\, d_1)\, ,
\cdots \, , (r_n\, d_n)$ satisfying the following conditions:
\begin{enumerate}
\item for any $(E\, ,\theta)\, \in\, {\mathbb L}$, there is filtration of
subbundles
\begin{equation}\label{e5}
0\, =\, E_0\, \subset\, E_1\, \subset\, E_2 \, \subset\, \cdots 
\, \subset\, E_{n-1} \, \subset\, E_n\,=\, E
\end{equation}
such that $\text{rank}(E_i/E_{i-1}) \,=\, r_i$ and
$\text{degree}(E_i/E_{i-1}) \,=\, d_i$ for all $1\,\leq\, i\, \leq\, n$, and

\item $\theta(E_i)\, \subset\, E_{i-1}\otimes K_X$ for all $1\,\leq\, i\, \leq\, n$.
\end{enumerate}
(See the proof of Theorem 5.3 in \cite[p. 228]{BR}.)

As before, $z\,=\, (E\, ,\theta)\, \in\, {\mathbb L}$ is a smooth point. We will
describe the tangent space $T_z {\mathbb L}$. Consider the filtration in \eqref{e5}. Let
$$
\End^p(E)\, \subset\, \End(E)
$$
be the subbundle defined by the condition that $\End^p(E)(E_i)\, \subset\, E_i$
for all $1\,\leq\, i\, \leq\, n$. Let
$$
\End^n(E)\, \subset\, \End^p(E)
$$
be the subbundle defined by the condition that $\End^n(E)(E_i)\, \subset\, E_{i-1}$
for all $1\,\leq\, i\, \leq\, n$. Let ${\mathcal D}^\bullet$ be the two-term complex
$$
{\mathcal D}^\bullet \, :\, D^0\,:=\,\End^p(E)\,
\stackrel{\text{ad}(\theta)}{\longrightarrow}\, D^1\,:=\,\End^n(E)\otimes K_X\, ,
$$
where $\text{ad}(\theta)$ is the homomorphism in \eqref{e1}. Note that
the condition $\theta(E_i)\, \subset\, E_{i-1}\otimes K_X$, $1\,\leq\, i\, \leq\, n$,
ensures that $\text{ad}(\theta)(\End^p(E))\, \subset\, \End^n(E)\otimes K_X$. The
inclusion of complexes
$$
\xymatrix{
D^{0}\ar[d]\ar[r]^{ad(\theta)} & D^{1}\ar[d]\\
\End(E)\ar[r]^{\hspace*{-10pt}ad(\theta)}& \End(E)\otimes K_{X}
}
$$
induces a homomorphism of hypercohomologies
$$
\varphi\, :\, {\mathbb H}^1({\mathcal D}^\bullet)\,\longrightarrow\,
{\mathbb H}^1({\mathcal C}^\bullet)
$$
(see \eqref{e1}). We have
$$
T_z {\mathbb L}\,=\, {\mathbb H}^1({\mathcal D}^\bullet)
$$
\cite[p. 229, Proposition 5.7]{BR} (see also the proof of Proposition (4.9) in
\cite[p. 662]{La}), and the differential $T_z {\mathbb L}\,\longrightarrow\,
T_z M_H(r,d)$ of the inclusion map ${\mathbb L}\, \hookrightarrow \, M_H(r,d)$
coincides with the above homomorphism $\varphi$. 

We note that
\begin{itemize}
\item $[\End^p(E)\, ,\End^n(E)]\, \subset\, \End^n(E)$, and

\item the homomorphism
$$\text{trace}\, :\, \End(E)\, \longrightarrow\, {\mathcal O}_X
$$
vanishes identically on $\End^n(E)$.
\end{itemize}
Therefore, for the homomorphism $\omega(E,\theta)$ in \eqref{e3}, we have
$$
\omega(E,\theta)\circ \varphi\,=\, 0\, .
$$
Hence $\iota^*\omega\,=\, 0$. In other words, $\mathbb L$ is
a Bohr--Sommerfeld Lagrangian.
\end{proof}

\subsection{Principal Higgs bundles}\label{3.1}

Theorem \ref{thm1} generalizes to the moduli spaces of $G$--Higgs bundles of fixed
topological type, where $G$ is any reductive complex algebraic group.

Let $G$ be a complex reductive group. The topological type of $G$--Higgs bundles
on $X$ are parametrized by $\pi_1(G)$. For any $\delta\, \in\, \pi_1(G)$, let
$M_H(G,\delta)$ denote the moduli space of $G$--Higgs bundles on $X$ of
topological type $\delta$.

We have the following:

\begin{theorem}\label{thm2}
The compact Bohr--Sommerfeld Lagrangians in $M_H(G,\delta)$ are precisely the
irreducible components of the nilpotent cone of $M_H(G,\delta)$.
\end{theorem}

The proof is same as for the case of Higgs vector bundles. The first proof of Theorem
\ref{thm1} goes through without any change. The second proof of Theorem \ref{thm1}
also works. Indeed, for this the only point to note is that if $\mathfrak p$ is the Lie
algebra of a parabolic subgroup $P$ of $G$, and $\mathcal B$ is any $G$--invariant
nondegenerate symmetric bilinear form on the Lie algebra of $G$, then the annihilator
of $\mathfrak p$, with respect to $\mathcal B$, is the Lie algebra of the unipotent
radical of $P$.

\subsection{Parabolic Higgs bundles}\label{3.2}

Theorem \ref{thm1} generalizes to the moduli spaces of parabolic Higgs bundles. In this
case the assumption that $g\, \geq\, 2$ is not needed. It is enough to assume that there
are stable parabolic Higgs bundles. 

In this case the deformation complex is 
$$
\mathcal{A}^{\bullet}\,:\, \PE(E)\,\longrightarrow\, \SPE(E)\otimes K(D)
$$
where
\begin{itemize}
\item $\PE(E)$ is the bundle of endomorphism of $E$ which preserves the quasiparabolic 
filtrations,

\item $\SPE(E)\, \subset\, \PE(E)$ is the subsheaf defined by the nilpotent endomorphisms
with respect to the quasiparabolic
filtrations, and

\item $K(D)\,=\, K_{X}\otimes 
\mathcal{O}_{X}(D)$ with $D$ being the divisor provided by the parabolic data.
\end{itemize}

Imitating $q$ in \eqref{e2}, consider the homomorphism 
$$
\widetilde{q}\,:\, \mathbb{H}^{1}(\mathcal{A}^{\bullet})\,\longrightarrow\,
H^{1}(X,\PE(E))\, .
$$
Finally in this case Serre duality for parabolic bundles provides the following morphism
$$
\widetilde{{\mathcal S}{\mathcal D}}\,:\,
H^{1}(X,\PE(E))\otimes H^{0}(X,\SPE(E)\otimes K(D))\,\longrightarrow\, \mathbb{C}\, .
$$
Now we define the algebraic one form for parabolic Higgs bundles as
$$
\omega(E,\theta)\,:\,\mathbb{H}^{1}(\mathcal{A}^{\bullet})
\,\longrightarrow \,\mathbb{C}\,, \ ~ \ \alpha\,\longmapsto\,
\widetilde{{\mathcal S}{\mathcal D}}(\widetilde{q}(\alpha)\otimes \theta)\, .
$$
Once the set-up is built, Theorem \ref{thm1} has a straightforward generalization.

Theorem \ref{thm1} also generalizes to parabolic
analog of principal Higgs bundles; see \cite{BBN}, \cite{Bo}, \cite{PR}, \cite{He}
for parabolic analog of principal Higgs bundles.

\section*{Acknowledgements}

We thank the referee for helpful comments. This work was supported by a Marie Curie
International Research Staff Exchange Scheme Fellowship within the 7th European
Union Framework Programme (FP7/2007-2013) under grant agreement
no. 612534, project MODULI - Indo European Collaboration. The second
author was partly supported by the center of excellence grant
'Center for Quantum Geometry of Moduli Spaces' from the Danish
National Research Foundation (DNRF95). The first author is supported
by the J. C. Bose Fellowship.

%%%%%%%%%%%%%%%%%%%%%%%%%%%%%%%%%%%%%%%%%%%%%%%%%%%%%%%%%%%%%%%%%

\end{document}